\documentclass{amsart} 

\usepackage[breaklinks=true, pdfauthor={Martin Weilandt}, pdftitle={Suborbifolds, quotients and transversality}, pdfkeywords={orbifold, suborbifold, embedding, transversality}]{hyperref}

\usepackage{amsmath}  
\usepackage{amsfonts} 
\usepackage{amssymb}  
\usepackage{amsthm}
\usepackage{enumerate}
\usepackage{fancyhdr} 
\usepackage[all]{xy}
\usepackage{tensor}
  
\newcommand{\Z}{\mathbb{Z}}
\newcommand{\R}{\mathbb{R}}
\newcommand{\C}{\mathbb{C}}
\renewcommand{\O}{\mathcal{O}}

\DeclareMathOperator{\Iso}{Iso}

\DeclareMathOperator{\rk}{rk}
\DeclareMathOperator{\gr}{gr}

\newcommand{\co}{\colon\thinspace}

\theoremstyle{plain}
\newtheorem{theorem}{Theorem}[section]
\newtheorem{lemma}[theorem]{Lemma}
\newtheorem{corollary}[theorem]{Corollary}
\newtheorem{proposition}[theorem]{Proposition}

\theoremstyle{definition}
\newtheorem{definition}[theorem]{Definition}

\newtheorem{example}[theorem]{Example}

\theoremstyle{remark}
\newtheorem{remark}[theorem]{Remark}
\newtheorem*{remark*}{Remark}

\newcommand*{\reg}{\text{reg}} 

\newenvironment{enumiii}{\begin{enumerate}[(i)]}{\end{enumerate}}

\title{Suborbifolds, quotients and transversality}
\author{Martin Weilandt}
\address{Universidade Federal de Santa Catarina, Departamento de Matem\'atica, Campus Universit\'ario Trindade, Florian\'opolis-SC, 88040-900, Brazil}
\email{martin.weilandt@ufsc.br}
\date{\today}

\numberwithin{equation}{section}

\begin{document}
\begin{abstract}
  Inspired by work of Borzellino and Brunsden, we generalize the notion of a submanifold identifying a natural and sufficiently general condition which guarantees that a subset of an (effective) orbifold carries itself a canonical induced orbifold structure. We illustrate the strength of this approach generalizing typical constructions of submanifolds to the orbifold setting using embeddings, proper group actions and the idea of transversality.
\end{abstract}
\date{\today}
\maketitle

\section{Introduction}
An orbifold is a topological space equipped with an atlas which locally provides homeomorphisms to quotients of some smooth manifold by a finite group of diffeomorphisms. Even though they provide a rather natural generalization of manifolds, there is no consensus on how to generalize the notion of a submanifold to the orbifold setting. After some remarks on group actions in Section \ref{sec:inv-sub} we give such a generalization, based on the idea of a ``saturated suborbifold'' in \cite{zbMATH05546218}, in Section \ref{sec:suborbifolds}, verify that a suborbifold is in fact an orbifold in a canonical manner and give some natural examples and constructions of suborbifolds. We also introduce the special cases of full and embedded suborbifolds, which have already been considered in \cite{zbMATH06513434} using somewhat different characterizations. In Section \ref{sec:quotients} we consider a quotient $M/G$ of a manifold $M$ under a proper almost free action of a Lie group $G$ and give conditions under which a submanifold $N\subset M$ and a subgroup $H\subset G$ define a suborbifold $N/H\subset M/G$.
In Section \ref{sec:alt} we provide alternative characterizations of full and embedded suborbifolds to clarify the relation between the terms in this paper and the suborbifold types from \cite{zbMATH06513434}. 
In Section \ref{sec:transversality} we define the notion of transverse suborbifolds and generalize classical results on transverse submanifolds and  maps to the setting of suborbifolds.
  
All our manifolds are Hausdorff, second countable, smooth without boundary and of (constant) finite dimension (though not necessarily connected). By a submanifold we always mean an embedded submanifold. Note that even though we frequently refer to \cite{zbMATH06513434}, we do not assume familiarity with that reference (or any other earlier works on ``suborbifolds'').

\section{Invariant submanifolds}
\label{sec:inv-sub}
Before coming to the definition of a suborbifold, we have to consider certain invariant submanifolds in the context of group actions.
Recall that given a Lie group $G$, a $G$-manifold is a manifold equipped with a smooth (left-) $G$-action. The following two definitions are central for our theory. (Although the conditions below are common in works on group actions, we could not find any name for them in this general form.)

The condition on the orbits in the following definition has already been used in  \cite[Definition 7]{zbMATH06513434} to define the notion of a ``saturated suborbifold''.

\begin{definition}
  \label{def:H-sub}
  Let $H$ be a closed subgroup of Lie group $G$ and let $M$ be a $G$-manifold.
  An \emph{$H$-submanifold} of $M$ is a submanifold $N$ of $M$ such that $N\cap Gx=Hx$ for every $x\in N$.

\end{definition}

\begin{remark}
  \label{remark:sat}
  The condition $N\cap Gx=Hx~\forall x\in N$ in the definition above is easily seen to be equivalent to the following characterization: $N$ is $H$-invariant and if $g\in G$ and $x\in N$ such that $gx\in N$, then there is $h\in H$ such that $hx=gx$.
\end{remark}

The following definition is closely related to the definition of a ``full suborbifold'' in \cite[Definition 5]{zbMATH06513434} (compare Lemma \ref{lemma:stab}).

\begin{definition}
  \label{def:full}
  Let $H$ be a closed subgroup of a Lie group $G$ and let $M$ be a $G$-manifold.
  A \emph{full $H$-submanifold} of $M$ is an $H$-invariant submanifold $N$ of $M$ with the following property: if $g\in G$ and $x\in N$ such that $gx\in N$, then $g\in H$.
\end{definition}

Of course, every full $H$-submanifold is an $H$-submanifold by Remark \ref{remark:sat}. Also note that if $G$ acts freely, then the two notions agree.

We will need the following basic observation about full $H$-submanifolds.

\begin{lemma}
  \label{lemma:stab}
  Let $H$ be a closed subgroup of a Lie group $G$, let $M$ be a $G$-manifold and let $N$ be an $H$-submanifold of $M$. Then $N$ is a full $H$-submanifold if and only if for every $x\in N$ the stabilizers $H_x$ and $G_x$ agree.
\end{lemma}

\begin{proof}
  First assume that $N$ is a full $H$-submanifold and let $x\in N$ and $g\in G_x$. Since $gx=x\in N$, we have $g\in H$. Hence $G_x\subset H$ and we conclude $H_x=H\cap G_x=G_x$. Conversely, assume $H_x=G_x$ for every $x\in N$. If $g\in G$ and $x\in N$ such that $gx\in N$, then there is $h\in H$ such that $hx=gx$ and hence $g^{-1}h\in G_x=H_x$. In particular, $g\in H$. 
\end{proof}

If $N$ is an $H$-submanifold of a $G$-manifold $M$, then we would like to consider the quotient $N/H$ as a subspace of $M/G$ with the inclusion given by $\iota\co N/H\ni Hx\mapsto Gx\in M/G$. Even though $\iota$ is injective in this setting (as already pointed out in the proof of \cite[Theorem 1 (1)]{zbMATH06513434}), we need additional hypotheses to guarantee that it is a topological embedding.

\begin{lemma}
  \label{lemma:embedding}
  Let $H$ be a subgroup of a finite group $G$ and let $N$ be a closed $H$-submanifold of a $G$-manifold $M$. Then the map $\iota\co N/H\ni Hx\mapsto Gx\in M/G$ is a topological embedding.
\end{lemma}
\begin{proof}
  First note that $\iota$ is well-defined, since $H$ is a subgroup of $G$. To see that $\iota$ is injective, let $x,y\in N$ such that $Gx=Gy$ and observe that Definition \ref{def:H-sub} implies $Hx=N\cap Gx=N\cap Gy=Hy$. $\iota$ is easily seen to be continuous. To conclude that $\iota$ is a homeomorphism onto its image, denote the (open!) quotient maps by $\pi_G\co M \to M/G$ and $\pi_H: N \to N/H$ and let $U$ be open in $N/H$. We have to show that $\iota(U)=\pi_G(\pi_H^{-1}(U))$ is open in $\iota(N/H)$.

  Let $x$ be in $\pi_H^{-1}(U)$. We shall construct an open subset $W$ of $\iota(N/H)$ such that $\pi_G(x) \in W \subset \iota(U)$. Let $V$ be an open subset of $M$ such that $\pi_H^{-1}(U) = V \cap N$ and write $G_x^N=\{ g \in G;~ gx \in N\}$. Since $N$ is an $H$-submanifold of $M$, for each $g \in G_x^N$ we have $gx \in \pi_H^{-1}(U) \subset V$ and hence (since $G$ is finite) $x$ has an open neighborhood $V^\prime$ such that $g V^\prime \subset V$ for every $g \in G_x^N$. Since $N$ is closed and $G$ is finite, we can diminish $V^\prime$ if necessary to guarantee that $g V^\prime \cap N$ is empty for each $g\in G\setminus G_x^N$. Then the set
  \[W := \pi_G(GV^\prime \cap N) = \pi_G(GV^\prime) \cap \pi_G(N) = \pi_G(GV^\prime) \cap \iota(N/H)\]
  has the desired properties. 
\end{proof}

Even though $\iota$ in the lemma above need not be an embedding if $N$ is not assumed to be closed, in this case one can still construct a canonical embedding between appropriate quotients of open subsets of $N$ and $M$ by applying the following lemma.

\begin{lemma}
  \label{lemma:closed}
  Let $G$ be a finite group, let $H\subset G$ be a subgroup and let $N$ be an $H$-submanifold of a $G$-manifold $M$. Then for every $x\in N$ there is an open connected $G_x$-invariant neighborhood $U\subset M$ and an open connected $H_x$-invariant neighborhood $V\subset N$ such that $V$ is a closed $H_x$-submanifold of $U$.
\end{lemma}
\begin{proof}
  Since $N\subset M$ is a submanifold, there is an open neighborhood $U\subset M$ of $x$ such that $U\cap N$ is closed in $U$. Diminishing $U$ if necessary, we can assume that $U$ is $G_x$-invariant, $U\cap gU=\emptyset$ for every $g\in G\setminus G_x$ and that $U$ is connected.

  Note that $U\cap N$ is an $H_x$-submanifold of the $G_x$-manifold $U$: Let $y\in U\cap N$ and $g\in G_x$ such that $gy\in U\cap N$. Since $N$ is an $H$-submanifold of $M$, there is $h\in H$ such that $gy=hy$. Since $y, hy\in U$, we have $h\in G_x$ and conclude that $h\in H\cap G_x=H_x$.

  Finally, let $V$ be the connected component of $x$ in $U\cap N$.
\end{proof}

Since our proof of Proposition \ref{proposition:suborbis} uses a geometrical argument, we will also need the following lemma. (See \cite[chapter 2]{zbMATH01626771} for the theory of intrinsic metric spaces.)
\begin{lemma}
  \label{lemma:metrics}
  Let $G$ be a finite group acting isometrically on a Riemannian manifold $(M,g)$, let $H$ be a subgroup of $G$, let $N$ be a closed $H$-submanifold of $M$ and let $g^\prime$ denote the Riemannian metric on $N$ given by the pull-back of $g$ via $N\hookrightarrow M$. Then the quotient metric on $N/H$ induced by $g^\prime$ and the intrinsic metric on $N/H$ induced by the quotient metric on $M/G$ (induced by $g$) coincide.
\end{lemma}
\begin{proof}
  Let $d=d_g$ denote the intrinsic metric (distance function) on $M$ induced by $g$ and let $e=e_{g^\prime}$ denote the intrinsic metric (distance function) on $N$ induced by $g^\prime$. Let $\bar{e}$ denote the quotient metric on $N/H$ induces by $e$ and let $\widehat{d}$ denote the intrinsic metric on $N/H$ induced by the quotient metric $\bar{d}$ on $M/G$.

  Let $Hx,Hy\in N/H$. Then
  \[\bar{e}(Hx,Hy)=\min_{h\in H}e(x,hy)=\inf_{c\in \mathcal{C}}\sup_{\mathcal{P}(c)}\sum_id(c(t_i),c(t_{i+1})),\]
  where $\mathcal{C}$ denotes the collection of continuous paths in $N$ from $x$ to some point in the orbit $Hy$ and $\mathcal{P}(c)$ denotes the collection of partitions $a\le t_0\le t_1\le\cdots\le t_k=b$ of the domain $[a,b]$ of the curve $c$.
  On the other hand (with $\pi\co N\to N/H$ denoting the quotient map),
  \[\widehat{d}(Hx,Hy)=\inf_{c\in \mathcal{C}}\sup_{\mathcal{P}(c)}\sum_i\bar{d}(\pi\circ c(t_i),\pi\circ c(t_{i+1}))
    =\inf_{c\in \mathcal{C}}\sup_{\mathcal{P}(c)}\sum_i\min_{g\in G}d(c(t_i), g c(t_{i+1})).\]
  To conclude that $\bar{e}(Hx,Hy)=\widehat{d}(Hx,Hy)$, fix a curve $c\co[a,b]\to N$ in $\mathcal{C}$. (If $\mathcal{C}$ is empty, then $\bar{e}(Hx,Hy)=\infty=\widehat{d}(Hx,Hy)$.) For each $t\in [a,b]$ there is $\delta(t)\in (0,\varepsilon)$ such that $B_{2\delta(t)}(c(t))\cap Gc(t)\setminus\{c(t)\}=\emptyset$ (where $B$ stands for an open ball in $M$ with respect to $d$). Since the image $c([a,b])$ is compact, there is a partition $a=t_0\le t_1\le\ldots\le t_k=b$ such that, with $J_i$ denoting the connected component of $B_{\delta(t_i)/2}(c(t_i))\cap c([a,b])$ containing $c(t_i)$, we have $\bigcup_i J_i=c([a,b])$. For each $0\le i\le k-1$ write $\delta_i:=\delta(t_i)$ and note that $d(c(t_i),c(t_{i+1}))<(\delta_i+\delta_{i+1})/2$. Now fix $i$ and first assume that $\delta_i\le\delta_{i+1}$. If $g\in G\setminus G_{c(t_{i+1})}$, then
  \begin{align*}
    d(c(t_i),gc(t_{i+1}))&\ge d(c(t_{i+1}),gc(t_{i+1}))-d(c(t_{i+1}),c(t_i))> 2\delta_{i+1}-(\delta_{i}+\delta_{i+1})/2\\
    &\ge(\delta_i+\delta_{i+1})/2>d(c(t_i),c(t_{i+1}))
  \end{align*}
  and hence $\min_{g\in G}d(c(t_i), g c(t_{i+1}))=d(c(t_i),c(t_{i+1}))$. If $\delta_i>\delta_{i+1}$, then an analogous calculation shows that $d(gc(t_i),c(t_{i+1}))>d(c(t_i),c(t_{i+1}))$ for every $g\in G\setminus G_{c(t_{i})}$ and hence
  \[\min_{g\in G}d(c(t_i), g c(t_{i+1}))=\min_{g\in G}d(gc(t_i),c(t_{i+1}))=d(c(t_i),c(t_{i+1}))\]
  in this case as well. We conclude that $\bar{e}(Hx,Hy)=\widehat{d}(Hx,Hy)$.
\end{proof}

\section{Suborbifolds}
\label{sec:suborbifolds}
Orbifolds have first been considered in \cite{zbMATH03123695} and we refer the reader to \cite{MR1950941, MR2012261, zbMATH01059650} for more results on the compatibility conditions used in this work. Let $X$ be a Hausdorff space with second countable topology. An $n$-dimensional \emph{orbifold chart} on $X$ is a quadruple $(U,\widetilde{U}/\Gamma,\pi)$ in which $U\subset X$ is open, $\Gamma$ is a finite group, $\widetilde{U}$ is an $n$-dimensional connected manifold with a fixed smooth effective $\Gamma$-action and $\pi\co\widetilde{U}\to U$ is a continuous $\Gamma$-invariant map such that the induced map $\overline{\pi}\co \widetilde{U}/\Gamma\to U$ is a homeomorphism. An \emph{injection} between two charts $(V,\widetilde{V}/\Delta,\phi)$, $(U,\widetilde{U}/\Gamma,\pi)$ with $V\subset U$ is a smooth embedding $\lambda\co\widetilde{V}\to\widetilde{U}$ such that $\pi\circ\lambda=\phi$. It is known that for each such injection $\lambda$ there is a unique map $\overline{\lambda}\co\Delta\to\Gamma$ such that $\lambda(\gamma x)=\overline{\lambda}(\gamma)\lambda(x)$, which turns out to be a monomorphism. Two orbifold charts $(U_i,\widetilde{U}_i/\Gamma_i,\pi_i)$, $i=1,2$, on $X$ are called \emph{compatible in $p$} if there is an orbifold chart $(V,\widetilde{V}/\Delta,\phi)$ such that $p\in V\subset U_1\cap U_2$ and, for each $i=1,2$, an injection $\lambda_i\co(V,\widetilde{V}/\Delta,\phi)\to(U_i,\widetilde{U}_i/\Gamma_i,\pi_i)$. An $n$-dimensional orbifold \emph{atlas} on $X$ is a collection $\{(U_\alpha,\widetilde{U}_\alpha/\Gamma_\alpha,\pi_\alpha)\}_\alpha$ of orbifold charts on $X$ such that $X=\bigcup_\alpha U_\alpha$ and whenever $U_\alpha\cap U_\beta\neq\emptyset$, then $\pi_\alpha$ and $\pi_\beta$ are compatible in every $p\in U_\alpha\cap U_\beta$. Since compatibility in a fixed point $p\in X$ defines an equivalence relation, it is easily seen that every orbifold atlas is contained in a unique maximal one. An (effective) \emph{orbifold} is by definition a pair $\mathcal{O}=(X,\mathcal{A})$ of a space $X$ and a maximal atlas $\mathcal{A}$. Just as in the manifold setting we also denote the so-called \emph{underlying space} $X$ simply by $\mathcal{O}$.
Given an orbifold $\mathcal{O}$, the \emph{isotropy} of $p\in\mathcal{O}$, denoted by $\Iso(p)$, is the (well-defined) isomorphism class of any stabilizer $\Gamma_{\tilde{p}}$ with $(U,\widetilde{U}/\Gamma,\pi)$ a chart such that $p\in U$ and $\tilde{p}\in\pi^{-1}(p)$. The \emph{regular part} $\mathcal{O}^\reg$ of an orbifold $\mathcal{O}$ is given by all $p\in\mathcal{O}$ such that $\Iso(p)$ is trivial and an orbifold $\mathcal{O}$ is canonically identified with a manifold if and only if $\mathcal{O}=\mathcal{O}^\reg$. Given two orbifolds $\mathcal{O}_1$, $\mathcal{O}_2$ equipped with maximal atlases $\{(U_\alpha,\widetilde{U}_\alpha/\Gamma_\alpha,\pi_\alpha)\}_\alpha$, $\{(V_\beta,\widetilde{V}_\beta/\Delta_\beta,\phi_\beta)\}_\beta$, respectively, the \emph{product orbifold} $\mathcal{O}_1\times\mathcal{O}_2$ is just the product of the underlying spaces equipped with the orbifold structure containing the atlas $\{(U_\alpha\times V_\beta,(\widetilde{U}_\alpha\times\widetilde{V}_\beta)/(\Gamma_\alpha\times\Delta_\beta),\pi_\alpha\times\phi_\beta)\}_{\alpha,\beta}$

We can now introduce the notion of a suborbifold. The definition below is inspired by \cite[Definition 2.13]{zbMATH05546218} and additional conditions in \cite{zbMATH06513434}. 
With respect to \cite{zbMATH06513434}, we note that locally our notion of a suborbifold corresponds to the idea of a ``saturated suborbifold'', our notion of a full suborbifold locally corresponds to the homonymous notion in \cite{zbMATH06513434} and our notion of an embedded suborbifold locally corresponds to the notion of a ``saturated'' and ``split'' suborbifold. We will make this relation more precise in Section \ref{sec:alt}. 

Note, however, that we do not assume in the definition that a suborbifold is an orbifold but instead establish this property in Proposition \ref{proposition:suborbis}. In Theorem \ref{theorem:embed-orbi} we will see that embedded suborbifolds as defined below correspond to images of orbifold embeddings in the sense of Definition \ref{def:immersion}.

\begin{definition}
  \label{def:suborbi}
  Let $\mathcal{O}$ be an orbifold, $0\le k\le\dim\mathcal{O}$ and $\mathcal{P}\subset\mathcal{O}$ a subset.
  \begin{enumiii}
  \item \label{def:suborbi:gen} $\mathcal{P}$ is a \emph{$k$-dimensional suborbifold} of $\mathcal{O}$ if for every $p\in \mathcal{P}$ there is an orbifold chart $(U,\widetilde{U}/\Gamma,\pi)$ of $\mathcal{O}$ with the property that there is a subgroup $\Delta\subset\Gamma$ and a $k$-dimensional $\Delta$-submanifold $\widetilde{V}\subset\widetilde{U}$ such that 
 $\pi(\widetilde{V})$ is an open neighborhood of $p$ in $\mathcal{P}$.
\item \label{def:suborbi:full} $\mathcal{P}$ is a \emph{full $k$-dimensional suborbifold} of $\mathcal{O}$ if for every $p\in \mathcal{P}$ there are $(U,\widetilde{U}/\Gamma,\pi)$, $\Delta$ and $\widetilde{V}$ as in (\ref{def:suborbi:gen}) such that $\widetilde{V}$ is a full $\Delta$-submanifold.
\item \label{def:suborbi:embedded} $\mathcal{P}$ is an \emph{embedded $k$-dimensional suborbifold} of $\mathcal{O}$ if for every $p\in \mathcal{P}$ there are $(U,\widetilde{U}/\Gamma,\pi)$, $\Delta$ and $\widetilde{V}$ as in (\ref{def:suborbi:gen}) such that $\widetilde{V}$ is connected and the action of $\Delta$ on $\widetilde{V}$ is effective.
  \end{enumiii}
\end{definition}

\begin{remark}
  Lemma \ref{lemma:closed} implies that we obtain an equivalent definition adding the condition that each $\widetilde{V}$ 
  above be connected and closed in $\widetilde{U}$. (It is straightforward to see that the additional conditions in (\ref{def:suborbi:full}) and (\ref{def:suborbi:embedded}) above are preserved by the construction based on that lemma.)
\end{remark}

Each of these three notions of a suborbifold will turn out to be relevant in a certain context and in Section \ref{sec:alt} we verify examples from \cite{zbMATH06513434} which show that these classes of suborbifolds are mutually disjoint. For the moment just note that if $\mathcal{O}$ happens to be a manifold, then all of the definitions above coincide and correspond to the usual notion of a submanifold. 

Before coming to concrete examples, we show the crucial result that every suborbifold is an orbifold in a canonical way.

\begin{proposition}
  \label{proposition:suborbis}
  Let $\mathcal{O}$ be an orbifold, $0\le k\le\dim\mathcal{O}$ and $\mathcal{P}\subset\mathcal{O}$ a subset.
  \begin{enumerate}[(i)]
  \item If $(U,\widetilde{U}/\Gamma,\pi)$ is a chart on $\mathcal{O}$ and $\widetilde{V}$ and $\Delta$ are as in Definition \ref{def:suborbi} (\ref{def:suborbi:gen}) with $\widetilde{V}$ connected and closed in $\widetilde{U}$, $K$ is the kernel of the $\Delta$-action on $\widetilde{V}$ and $\phi:=\pi_{|\widetilde{V}}\co \widetilde{V}\to V:=\pi(\widetilde{V})$, then $(V,\widetilde{V}/(\Delta/K),\phi)$ is a $k$-dimensional orbifold chart on the second countable Hausdorff space $\mathcal{P}$.
  \item \label{proposition:suborbis2} If $\mathcal{P}\subset\mathcal{O}$ is a $k$-dimensional suborbifold, then it carries a canonical
    $k$-dimensional orbifold structure. This structure contains every chart\newline $(V,\widetilde{V}/(\Delta/K),\phi)$ as defined in (i).
  \end{enumerate}
\end{proposition}
\begin{proof}
  First note that $\mathcal{P}$ is Hausdorff and second countable as a topological subspace of $\mathcal{O}$.
  To see (i) first observe that $\phi\co \widetilde{V}\to V$ is invariant under the smooth effective action of the finite group $\Delta/K$. The induced map $\overline{\phi}\co \widetilde{V}/(\Delta/K)\to V$ is a homeomorphism, since $\widetilde{V}/(\Delta/K)=\widetilde{V}/\Delta$ is a topological subspace of $\widetilde{U}/\Gamma$ (by Lemma \ref{lemma:embedding}) such that $\overline{\pi}(\widetilde{V}/(\Delta/K))=V$. Since $\overline{\pi}$ is a homeomorphism, so is $\overline{\phi}=\overline{\pi}_{|\widetilde{V}/(\Delta/K)}\co \widetilde{V}/(\Delta/K)\to V$.

  We shall now show (ii):
  Let
  $\mathfrak{A}=\{(U_\alpha,\widetilde{U}_\alpha/\Gamma_\alpha,\pi_\alpha)\}_\alpha$
  be an atlas of $\mathcal{O}$ such that for each $\alpha$ there is a subgroup $\Delta_\alpha\subset\Gamma_\alpha$ and a closed
  connected $k$-dimensional $\Delta_\alpha$-submanifold $\widetilde{V}_\alpha\subset\widetilde{U}_\alpha$ such that $\pi_\alpha(\widetilde{V}_\alpha)$ is an open subset of $\mathcal{P}$ and assume that $\mathcal{P}=\bigcup_\alpha \pi_\alpha(\widetilde{V}_\alpha)$.

  Moreover, let
  $\phi_\alpha:={\pi_\alpha}_{|\widetilde{V}_\alpha}\co \widetilde{V}_\alpha\to
  V_\alpha:=\pi_\alpha(\widetilde{V}_\alpha)$ and let $K_\alpha$ denote the
  kernel of the $\Delta_\alpha$-action on
  $\widetilde{V}_\alpha$. From (i) we know that $(V_\alpha,\widetilde{V}_\alpha/(\Delta_\alpha/K_\alpha),\phi_\alpha)$ is a $k$-dimensional orbifold chart on the topological space $\mathcal{P}$.

  We are left to show that
  the collection $\mathfrak{B}=\{(V_\alpha,\widetilde{V}_\alpha/(\Delta_\alpha/K_\alpha),\phi_\alpha)\}_\alpha$
  is a $k$-dimensional orbifold atlas for $\mathcal{P}$. Consider two charts
  $(V_i,\widetilde{V}_i/(\Delta_i/K_i),\phi_i)$,
  $i=1,2$, in $\mathfrak{B}$ and $p\in V_1\cap V_2$ and let
  $(U_i,\widetilde{U}_i/\Gamma_i,\pi_i)$,
  $i=1,2$, be the corresponding charts in
  $\mathfrak{A}$. Since $\mathfrak{A}$ is an atlas of $\mathcal{O}$,
  there is an $n$-dimensional orbifold chart
  $(W,\widetilde{W}/\Sigma,\eta)$ on $\mathcal{O}$ (not necessarily in $\mathfrak{A}$)
  around $p$ and injections $\lambda_i\co
  \widetilde{W}\to\widetilde{U}_i$, $i=1,2$. Let $\widetilde{p}\in\eta^{-1}(p)$. 
  Given an open neighborhood $V^\prime\subset V_1\cap V_2$ of $p$, each $\widetilde{V}^\prime_i:=\phi_i^{-1}(V^\prime)$ is open in $\widetilde{V}_i$ and $\Delta_i$-invariant and hence a $\Delta_i$-submanifold of $\widetilde{U}_i$. Choosing $V^\prime$ sufficiently small and composing each $\lambda_i$ with an appropriate element of $\Gamma_i$, we can guarantee that $\lambda_i(\widetilde{p})\in\widetilde{V}^\prime_i\subset \lambda_i(\widetilde{W})$. 
   Setting $\phi_i^\prime:={\phi_i}_{|\widetilde{V}^\prime}$, we obtain the following commutative diagram (where $\hookrightarrow$ stands for the inclusion).

    \begin{center}
      \makebox[0pt]{
        \xymatrix{
          \widetilde{V}^\prime_1 \ar[d]^{\phi^\prime_1}\ar@{^{(}->}[r]^{\iota_1}&\widetilde{U}_1\ar[d]_{\pi_1}&\ar[l]_{\lambda_1}\widetilde{W}\ar[d]^{\eta}\ar[r]^{\lambda_2}&\widetilde{U}_2\ar[d]^{\pi_2}& \ar@{_{(}->}[l]_{\iota_2}\widetilde{V}^\prime_2\ar[d]^{\phi^\prime_2}\\
          V^\prime\ar@{^{(}->}[r] &U_1&\ar@{_{(}->}[l] W\ar@{^{(}->}[r]&U_2&\ar@{_{(}->}[l] V^\prime
        }
      }
    \end{center}
Now equip $\widetilde{W}$ with a $\Sigma$-invariant Riemannian metric $g$, $W\approx\widetilde{W}/\Sigma$ with the quotient metric $\overline{d}=\overline{d_g}$ induced by $g$ and $V^\prime\subset W$ with the intrinsic metric $\widehat{d}$ induced by $\overline{d}$. Identifying each $\widetilde{V}_i^\prime/\Delta_i=\widetilde{V}_i^\prime/(\Delta_i/K_i)$ with $V^\prime$ (via ${\overline{\phi}_i^\prime}$), Lemma \ref{lemma:metrics} implies that the metric induced on $\widetilde{V}_i^\prime/\Delta_i$ by $\widehat{d}$ coincides with the quotient metric induced by the Riemannian metric $h_i:=({\lambda_i}^{-1}\circ\iota_i)^\ast g$ on $\widetilde{V}_i^\prime$. In particular, ${\overline{\phi}_2^\prime}^{-1}\circ\overline{\phi}_1^\prime\co \widetilde{V}_1^\prime/\Delta_1\to\widetilde{V}_2^\prime/\Delta_2$ is an isometry with respect to the quotient metrics. By the proof of \cite[Lemma 5.2.2]{claudio}, there is $\varepsilon>0$ such that each open ball $B_\varepsilon(\lambda_i(\widetilde{p}))\subset (\widetilde{V}_i^\prime,h_i)$ is a full ${\Delta_i}_{\widetilde{p}}$-submanifold and there is an isometry $\mu\co B_\varepsilon(\lambda_1(\widetilde{p}))\to B_\varepsilon(\lambda_2(\widetilde{p}))$ satisfying $\phi_2^\prime\circ\mu=\phi_1^\prime$ on $B_\varepsilon(\lambda_1(\widetilde{p}))$. Since the restriction of $\phi_1^\prime$ to $B_\varepsilon(\lambda_1(\widetilde{p}))$ is an orbifold chart around $p$, $\phi_1$ and $\phi_2$ are compatible orbifold charts in $p$.

Since $\phi_1$ and $\phi_2$ were arbitrary charts from $\mathfrak{B}$, we conclude that $\mathfrak{B}$ is an orbifold atlas on the set $\mathcal{P}$.
  \end{proof}

  
  \begin{example}
    \label{example:suborbis}
  \begin{enumiii}
  \item Let $D=\{x\in\R^2;~\|x\|<1\} \subset\R^2$ be the open disk and let $R(\theta)$ denote the (positive) rotation around the origin by the angle $\theta$. Consider $\widetilde{U}=D$, $\Gamma=\langle R(\pi/2)\rangle$ and $\pi\co \widetilde{U}\to U:=D/\Gamma$ the canonical projection. Then $\mathcal{O}=(U,\mathfrak{A})$ with atlas given by $\mathfrak{A}=\{(U,\widetilde{U}/\Gamma,\pi)\}$ is an orbifold and the set $\mathcal{P}=((-1,1)\times\{0\})/\langle R(\pi)\rangle$ is a suborbifold of $\mathcal{O}$: we can just use $\{(U,\widetilde{U}/\Gamma,\pi)\}$ as above, $\widetilde{V}=(-1,1)\times\{0\}$ and $\Delta=\langle R(\pi)\rangle$ in Definition \ref{def:suborbi} (\ref{def:suborbi:gen}). Since $\Delta$ acts effectively on $\widetilde{V}$, the subset $\mathcal{P}\subset\mathcal{O}$ is an embedded suborbifold. Note, however, that $\widetilde{V}$ is not a full $\Delta$-submanifold, since $R(\pi/2)$ fixes $(0,0)\in\widetilde{V}$ but is not an element of $\Delta$. 
  \item \label{example:suborbis:pt}Given an orbifold $\mathcal{O}$, every subset containing just a point $p\in\mathcal{O}$ is a zero-dimensional full embedded suborbifold: If $(U,\widetilde{U}/\Gamma,\pi)$ is a chart on $\mathcal{O}$ around $p$ and $\widetilde{p}\in\pi^{-1}(p)$, then $\widetilde{V}:=\{\widetilde{p}\}$ is a full $G_{\widetilde{p}}$-submanifold and a $\{e\}$-submanifold of $\widetilde{U}$ and $\pi(\widetilde{V})=\{p\}$. Generalizing this example, it is straightforward to verify that the $0$-dimensional suborbifolds of a fixed orbifold are precisely the discrete subsets and that each such discrete suborbifold is full and embedded.
  \item A subset $\mathcal{P}$ of an $n$-dimensional orbifold $\mathcal{O}$ is an $n$-dimensional suborbifold if and only if $\mathcal{P}$ is open. Each such open suborbifold is full and embedded. To verify these claims, first let $\mathcal{P}\subset\mathcal{O}$ be an $n$-dimensional suborbifold, $p\in\mathcal{P}$ and $(U,\widetilde{U}/\Gamma,\pi)$, $\Delta\subset\Gamma$ and $\widetilde{V}\subset\widetilde{U}$ be as in Definition \ref{def:suborbi} (\ref{def:suborbi:gen}). Since $\pi(\widetilde{V})$ is open in $U$ and contains $p$, the set $\mathcal{P}$ is open in $\mathcal{O}$. Conversely, if $\mathcal{P}$ is an open subset of $\mathcal{O}$ and $p\in\mathcal{P}$, let $(U,\widetilde{U}/\Gamma,\pi)$ be a chart of $\mathcal{O}$ around $p$.  Diminishing $U$ if necessary, we can assume that $U\subset \mathcal{P}$ and conclude that $\mathcal{P}$ is an $n$-dimensional full embedded suborbifold of $\mathcal{O}$.
    \item If $\mathcal{O}_1$, $\mathcal{O}_2$ are two orbifolds and $\mathcal{P}_1\subset\mathcal{O}_1$, $\mathcal{P}_2\subset\mathcal{O}_2$ are suborbifolds, then $\mathcal{P}_1\times\mathcal{P}_2\subset\mathcal{O}_1\times\mathcal{O}_2$ is a suborbifold: If $(p_1,p_2)\in\mathcal{P}_1\times\mathcal{P}_2$ and, for each $i=1,2$, $(U_i,\widetilde{U}_i/\Gamma_i,\pi_i)$ is a chart on $\mathcal{O}_i$ around $p_i$ with associated subgroup $\Delta_i\subset\Gamma_i$ and $\Delta_i$-submanifold $\widetilde{V}_i$ of $\widetilde{U}_i$ as in Definition \ref{def:suborbi} (\ref{def:suborbi:gen}), then, setting $\Delta=\Delta_1\times\Delta_2$, we observe that $\widetilde{V}_1\times\widetilde{V}_2$ is a $\Delta$-submanifold of the $\Gamma_1\times\Gamma_2$-manifold $\widetilde{U}_1\times\widetilde{U}_2$ and $(\pi_1\times\pi_2)(\widetilde{V}_1\times\widetilde{V}_2)=\pi_1(\widetilde{V}_1)\times\pi_2(\widetilde{V}_2)$ is an open neighborhood of $(p_1,p_2)$ in $\mathcal{P}_1\times\mathcal{P}_2$.
    \item \label{example:suborbis:diag} Given an $n$-dimensional orbifold $\mathcal{O}$, the diagonal $\mathcal{D}:=\{(p,p);~p\in\mathcal{O}\}$ is an embedded $n$-dimensional suborbifold of $\mathcal{O}\times\mathcal{O}$ (compare \cite[Example 2.15]{zbMATH05546218}). This follows either by applying Proposition \ref{proposition:graph} below to the identity $\mathcal{O}\to\mathcal{O}$ or by the following direct argument: Given $(p,p)\in\mathcal{D}$, let $(U,\widetilde{U}/\Gamma,\pi)$ be a chart of $\mathcal{O}$ around $p$. Then $(U\times U,(\widetilde{U}\times\widetilde{U})/(\Gamma\times\Gamma),\pi\times\pi)$ is a chart on $\mathcal{O}\times\mathcal{O}$. It is straightforward to verify that the diagonal $\widetilde{V}:=\{(x,x);~x\in\widetilde{U}\}$ is a connected $n$-dimensional $\Delta$-submanifold of $\widetilde{U}\times\widetilde{U}$ with respect to $\Delta:=\{(\gamma,\gamma);~\gamma\in\Gamma\}$ (which acts effectively on $\widetilde{V}$) and that $(\pi\times\pi)(\widetilde{V})=\{(q,q);~q\in U\}=(U\times U)\cap\mathcal{D}$.
  \end{enumiii}
\end{example}




\begin{definition}
  \label{def:smooth}
  \begin{enumiii}
  \item \label{def:smooth:i} A continuous map $f\co\mathcal{O}_1\to\mathcal{O}_2$ between (the underlying spaces of) orbifolds is called \emph{smooth} if for every $p\in\O_1$ there are charts $(U_i,\widetilde{U}_i/\Gamma_i,\pi_i)$, $i=1,2$, around $p$ and $f(p)$, respectively, a smooth map $\widetilde{f}\co \widetilde{U}_1\to\widetilde{U}_2$ and a homomorphism $\Theta\co\Gamma_1\to\Gamma_2$ such that $f\circ\pi_1=\pi_2\circ\widetilde{f}$ and $\widetilde{f}(\gamma x)=\Theta(\gamma)\widetilde{f}(x)$ for every $\gamma\in\Gamma_1$ and $x\in\widetilde{U}_1$.
  \item \label{def:smooth:ii} Let $f\co\mathcal{O}_1\to\mathcal{O}_2$ be smooth and $p\in\mathcal{O}_1$. The \emph{rank} of $f$ in $p$, denoted by $\rk f(p)$, is the rank of any lift $\widetilde{f}$ as in (\ref{def:smooth:i}) in some $\widetilde{p}\in\pi_1^{-1}(p)$.
  \end{enumiii}
\end{definition}

\begin{remark}
  Note that the compatibility conditions on $\mathcal{O}_1$ imply that the rank in (\ref{def:smooth:i}) is well-defined.
\end{remark}

  
We now verify that the graph of a smooth map is a suborbifold. Note that this statement and its proof below are similar to \cite[Proposition 3.8]{zbMATH05546218}. We include the result as a reference, since in the context of \cite{zbMATH05546218}, which does not impose any compatibility conditions on suborbifold charts, it is not clear if the graph is in fact an orbifold. Now using our more rigid condition on suborbifold charts (referred to as ``saturated'' in \cite{zbMATH06513434}) and our Proposition \ref{proposition:suborbis}, this is a direct consequence of the following proposition.


\begin{proposition}
  \label{proposition:graph}
  Let $f\co\mathcal{O}_1\to\mathcal{O}_2$ be a smooth map between orbifolds and $n=\dim\mathcal{O}_1$. Then the graph $\gr f$ of $f$ is an embedded $n$-dimensional suborbifold of $\mathcal{O}_1\times\mathcal{O}_2$. If the image of $f$ is contained in the regular part of $\mathcal{O}_2$, then $\gr f$ is a full suborbifold.
\end{proposition}
\begin{proof}
  Let $(p,f(p))\in\gr f$. Then there are charts $(U_i,\widetilde{U}_i/\Gamma_i,\pi_i)$, $i=1,2$, around $p$, $f(p)$, respectively, a smooth map $\widetilde{f}\co\widetilde{U}_1\to\widetilde{U}_2$ and a homomorphism $\Theta\co\Gamma_1\to\Gamma_2$ such that $\pi_2\circ\widetilde{f}=f\circ\pi_1$ and $\widetilde{f}(\gamma x)=\Theta(\gamma)\widetilde{f}(x)$. Now consider the chart $(U_1\times U_2,(\widetilde{U}_1\times\widetilde{U}_2)/(\Gamma_1\times\Gamma_2),\pi_1\times\pi_2)$ on $\mathcal{O}_1\times\mathcal{O}_2$ and note that the graph $\widetilde{V}$ of $\widetilde{f}$ is a connected $n$-dimensional submanifold of $\widetilde{U}_1\times\widetilde{U}_2$ and the graph $\Delta$ of $\Theta$ is a subgroup of $\Gamma_1\times\Gamma_2$. To see that $\widetilde{V}$ is a $\Delta$-submanifold of the $\Gamma_1\times\Gamma_2$-manifold $\widetilde{U}_1\times\widetilde{U}_2$, first note that it is easily seen to be $\Delta$-invariant (since $\widetilde{f}$ is $\Theta$-equivariant). Now let $(\gamma_1,\gamma_2)\in\Gamma_1\times\Gamma_2$ and $(x,\widetilde{f}(x))\in\widetilde{V}$ such that $(\gamma_1,\gamma_2)(x,\widetilde{f}(x))\in\widetilde{V}$. Then $\gamma_2\widetilde{f}(x)=\widetilde{f}(\gamma_1x)=\Theta(\gamma_1)\widetilde{f}(x)$ and hence $(\gamma_1,\gamma_2)(x,\widetilde{f}(x))=(\gamma_1,\Theta(\gamma_1))(x,\widetilde{f}(x)))$ with $(\gamma_1,\Theta(\gamma_1))\in\Delta$ and we obtain that $\widetilde{V}$ is a $\Delta$-submanifold of $\widetilde{U}_1\times\widetilde{U}_2$. Since the $\Gamma_1$-action on $\widetilde{U}_1$ is effective, so is the action of $\Delta$ on $\widetilde{V}$. Since $(\pi_1\times\pi_2)(\widetilde{V})$ is just the graph of $f_{|U_1}$ and hence equal to $(U_1\times U_2)\cap\gr f$, we conclude that $\gr f$ is an embedded suborbifold of $\mathcal{O}_1\times\mathcal{O}_2$

  If the image of $f$ is contained in $\mathcal{O}_2^\reg$, then with $(\gamma_1,\gamma_2)\in\Gamma_1\times\Gamma_2$ and $x\in\widetilde{U}_1$ as above, the relation $\gamma_2\widetilde{f}(x)=\Theta(\gamma_1)\widetilde{f}(x)$ implies $(\gamma_1,\gamma_2)=(\gamma_1,\Theta(\gamma_1))\in\Delta$. Hence $\gr f$ becomes a full suborbifold.
\end{proof}

\begin{definition}
  \label{def:immersion}
  \begin{enumiii}
  \item Let $f\co\mathcal{O}_1\to\mathcal{O}_2$ be a smooth map between orbifolds. $f$ is an \emph{immersion} if it has constant rank $\dim\mathcal{O}_1$. $f$ is a \emph{submersion} if it has constant rank $\dim\mathcal{O}_2$.
  \item An \emph{(orbifold) embedding} is an immersion $f\co\mathcal{O}_1\to\mathcal{O}_2$ which is a topological embedding (between the underlying spaces).
  \end{enumiii}
\end{definition}

\begin{example}
  \label{example:smooth}
  \begin{enumiii}
  \item \label{example:smooth:emb} Let $\mathcal{P}\subset\mathcal{O}$ be an embedded suborbifold in the sense of Definition \ref{def:suborbi} (\ref{def:suborbi:embedded}). Then the inclusion $\iota\co\mathcal{P}\hookrightarrow\mathcal{O}$ is an orbifold embedding, since it is a topological embedding and locally (using the notation from Proposition \ref{proposition:suborbis} (\ref{proposition:suborbis2}) and assuming $\Delta$ to act effectively on $\widetilde{V}$) we can use the inclusions $\widetilde{\iota}\co\widetilde{V}\hookrightarrow\widetilde{U}$ and $\Theta\co\Delta\hookrightarrow \Gamma$.
  \item \label{example:smooth:graph} In the setting of Proposition \ref{proposition:graph} the map $\mathcal{O}_1\ni p\mapsto (p,f(p))\in \mathcal{O}_1\times\mathcal{O}_2$ is an orbifold embedding as is easily revealed by a closer look at the proof of that proposition.
  \end{enumiii}
\end{example}

\begin{proposition}
  \label{proposition:embed-suborbi}
  If $f\co\mathcal{O}_1\to\mathcal{O}_2$ is an embedding and $\mathcal{P}$ is a suborbifold of $\mathcal{O}_1$, then $f(\mathcal{P})$ is a suborbifold of $\mathcal{O}_2$ of dimension $\dim\mathcal{P}$. If $\mathcal{P}$ is embedded, then so is $f(\mathcal{P})$.
\end{proposition}
\begin{proof}
  Given $p\in\mathcal{P}$, let $(U_i,\widetilde{U}_i/\Gamma_i,\pi_i)$, $i=1,2$, be charts around $p$ and $f(p)$, respectively, $\widetilde{f}\co\widetilde{U}_1\to\widetilde{U}_2$ an immersion and $\Theta\co\Gamma_1\to\Gamma_2$ a homomorphism such that $\pi_2\circ\widetilde{f}=f\circ\pi_1$ and $\widetilde{f}(\gamma x)=\Theta(\gamma)\widetilde{f}(x)$ for every $\gamma\in\Gamma_1$, $x\in\widetilde{U}_1$. Since every $\widetilde{p}\in\pi_1^{-1}(p)$ has an open neighborhood on which $\widetilde{f}$ is a smooth embedding, we can, after diminishing $U_1$ if necessary, assume that $\widetilde{f}$ is a smooth embedding. 

Again diminishing $U_1$ if necessary, we can assume that there is a subgroup $\Delta_1$ of $\Gamma_1$ and a $\Delta_1$-submanifold $\widetilde{V}_1\subset\widetilde{U}_1$ such that $\pi_1(\widetilde{V}_1)$ is an open neighborhood of $p$ in $\mathcal{P}$.

Then $\Delta_2:=\Theta(\Delta_1)$ is a subgroup of $\Gamma_2$ and $\widetilde{V}_2:=\widetilde{f}(\widetilde{V}_1)$ is a $\Delta_2$-invariant submanifold of $\widetilde{U}_2$ of dimension $\dim\mathcal{P}$. To see that $\widetilde{V}_2$ is a $\Delta_2$-submanifold, let $\gamma\in\Gamma_2$ and $y\in\widetilde{V}_2$ such that $\gamma y\in\widetilde{V}_2$. Let $x_1,x_2\in\widetilde{V}_1$ such that $y=\widetilde{f}(x_1)$ and $\gamma y=\widetilde{f}(x_2)$. Then 
\[f(\pi_1(x_2))=\pi_2(\widetilde{f}(x_2))=\pi_2(\gamma\widetilde{f}(x_1))=\pi_2(\widetilde{f}(x_1))=f(\pi_1(x_1)).\]
Since $f$ is injective, $\pi_1(x_2)=\pi_1(x_1)$ and hence there is $\gamma^\prime\in\Gamma_1$ such that $x_2=\gamma^\prime x_1$. Since $\widetilde{V}_1$ is a $\Delta_1$-submanifold, there is $\delta\in\Delta_1$ such that $\delta x_1=\gamma^\prime x_1=x_2$. Thus
\[\Theta(\delta)y=\Theta(\delta)\widetilde{f}(x_1)=\widetilde{f}(\delta x_1)=\widetilde{f}(x_2)=\gamma y\]
and we obtain that $\widetilde{V}_2$ is a $\Delta_2$-submanifold of $\widetilde{U}_2$. Since $f$ is a topological embedding, $\pi_2(\widetilde{V}_2)=f(\pi_1(\widetilde{V}_1))$ is open in $f(\mathcal{P})$. Since $p\in \mathcal{P}$ was arbitrary, we conclude that $f(\mathcal{P})$ is suborbifold of $\mathcal{O}_2$ of dimension $\dim\mathcal{P}$.

Now let $\mathcal{P}$ be an embedded suborbifold of $\mathcal{O}_1$. Then, in the local situation above, we can assume that $\Delta_1$ acts effectively on the connected manifold $\widetilde{V}_1$. Let $\delta\in\Delta_1$ such that $\Theta(\delta)y=y$ for every $y\in\widetilde{V}_2$. Then $\widetilde{f}(\delta x)=\widetilde{f}(x)$ for every $x\in\widetilde{V}_1$. Since $\widetilde{f}$ is injective and $\Delta_1$ acts effectively on $\widetilde{V}_1$, we obtain $\delta=e$ and hence $\Theta(\delta)=e$. We conclude that $\Delta_2=\Theta(\Delta_1)$ acts effectively on $\widetilde{V}_2=\widetilde{f}(\widetilde{V}_1)$. Since this argument holds around every $p\in\mathcal{P}$, the suborbifold $f(\mathcal{P})$ is embedded in $\mathcal{O}_2$.
\end{proof}

Joining Example \ref{example:smooth} (\ref{example:smooth:emb}) and Proposition \ref{proposition:embed-suborbi} (applied to the embedded suborbifold $\mathcal{P}=\mathcal{O}_1$), we obtain:

\begin{theorem}
  \label{theorem:embed-orbi}
  If $\mathcal{O}$ is an orbifold, a subset $\mathcal{P}\subset\mathcal{O}$ is an embedded suborbifold (in the sense of Definition \ref{def:suborbi} (\ref{def:suborbi:embedded})) if and only if there is an orbifold $\mathcal{O}^\prime$ and an orbifold embedding $f\co\mathcal{O}^\prime\to\mathcal{O}$ such that $\mathcal{P}=f(\mathcal{O}^\prime)$.
\end{theorem}

\begin{remark}
  Note that \cite{zbMATH06513434} also addresses the question of the relation between embeddings and suborbifolds, using definitions a bit different from ours. In light of Proposition \ref{proposition:split}, the theorem above resembles \cite[Theorem 1 (2)]{zbMATH06513434}. 
\end{remark}



\section{Quotients}
\label{sec:quotients}

We would now like to construct suborbifolds using global quotients. Of course, if $H$ is a subgroup of a finite group $G$ and $N$ is a closed $H$-submanifold of a $G$-manifold $M$, then $N/H$ is a suborbifold of $M/G$. To generalize this result to possibly infinite groups, we need the following classical definitions.

\begin{definition}
  Let $G$ be a Lie group acting on a manifold $M$. The action is called \emph{almost free} if every stabilizer $G_x$, $x\in M$, is finite.
\end{definition}

\begin{definition}
  Let $G$ be a Lie group acting smoothly and almost freely on a manifold $M$. A \emph{slice} at $x\in M$ is a full $G_x$-submanifold $S$ of $M$ such that $x\in S$ and $T_yM=T_yS\oplus T_y(Gy)$ for every $y\in S$.
\end{definition}

Now fix a Lie group $G$ acting smoothly, almost freely and properly on a manifold $M$. The slice theorem \cite[Theorem 3.35]{2009arXiv0901.2374A} (based on \cite{zbMATH03167836}) says that at every $x\in M$ there is a slice $S$ (and the proof in \cite{2009arXiv0901.2374A} makes it clear that $S$ can be chosen arbitrarily small). If $p\co M\to M/G$ denotes the quotient map, $K$ is the kernel of the $G_x$-action on $S$ and we identify $p(S)=S/G_x$ with the open subset $GS/G$ of $M/G$ via the homeomorphism induced by the $G$-equivariant diffeomorphism $G\times_{G_x}S\ni[g,x]\mapsto gx\in GS$ (\cite[Theorem 3.40]{2009arXiv0901.2374A}), then $(GS/G,S/(G_x/K),p_{|S})$ is an orbifold chart on $M/G$ of dimension $\dim M-\dim G$. Since all these charts are compatible at every intersection point (see the proof of \cite[Teorema 2.4.6]{will}), we obtain a canonical orbifold structure on $M/G$. (Alternatively, \cite[Section 2.2]{MR2012261} constructs the quotient orbifold $M/G$ without referring to the slice theorem.)


For the proof of Theorem \ref{theorem:quot-suborbi} we need the following lemma.

\begin{lemma}
  \label{lemma:C}
  Let $G$ be a Lie group and let $\Gamma\subset G$ be a non-empty finite subset. Then there is an open neighborhood $C$ of $e$ in $G$ with the following property: if $g_1,g_2\in C$ and $\gamma\in\Gamma$ such that $g_1^{-1}\gamma g_2\in\Gamma$, then $g_1^{-1}\gamma g_2=\gamma$.
\end{lemma}
\begin{proof}
  Given $\gamma\in\Gamma$, consider the continuous map $f_\gamma\co G\times G\ni (g_1,g_2)\mapsto g_1^{-1}\gamma g_2\in G$. Since $f_\gamma^{-1}(\Gamma\setminus\{\gamma\})$ is closed and does not contain $(e,e)$, there is an open neighborhood $C_\gamma$ of $e$ such that $(C_\gamma\times C_\gamma)\cap f_\gamma^{-1}(\Gamma\setminus\{\gamma\})=\emptyset$.

  The intersection $C:=\bigcap_{\gamma\in\Gamma}C_\gamma$ has the desired property.
\end{proof}


\begin{theorem}
\label{theorem:quot-suborbi}
Let $G$ be a Lie group acting smoothly, almost freely and properly on a manifold $M$. If $H$ is a closed subgroup of $G$, $N$ is an $H$-submanifold of $M$ and $T_xN\cap T_x(Gx)=T_x(Hx)$ for every $x\in N$, then $\iota\co N/H\ni Hx\mapsto Gx\in M/G$ is an injective immersion.
\end{theorem}

\begin{proof}
  First note that the smooth action of $H$ on $N$ is almost free because $H_x=G_x\cap H$ and proper since $H$ is closed in $G$. In particular, $N/H$ is an orbifold with the canonical structure given above Lemma \ref{lemma:C}. $\iota$ is injective since $N$ is an $H$-submanifold of $M$.

  To see that $\iota$ is an immersion, let $x\in N$. Then there is a slice $S\subset M$ at $x$ for the $G$-action on $M$. Let $C$ be an open neighborhood of $e$ in $G$ with the property from Lemma \ref{lemma:C} with respect to $G$ and $\Gamma=G_x$. This property guarantees that the restriction of the $G$-action on $M$ to $C\times S$ gives a smooth injective map $\mu\co C\times S\to M$, which is seen to be an immersion by an elementary calculation using that $T_yS\cap T_y(Gy)=\{0\}$ for every $y\in S$. For dimension reasons $\mu$ is a diffeomorphism onto its (open) image $CS$. Since $CS\cap N$ is open in $N$, there is a slice $T\subset CS\cap N$ at $x$ for the induced $H$-action on $N$. We write $p_H\co N\to N/H$ and $p_G\co M\to M/G$ for the canonical quotient maps and consider the charts $(HT/H,T/(H_x/L),{p_H}_{|T})$ on $N/H$ and
  $(GS/G,S/(G_x/K),{p_G}_{|S})$ on $M/G$ (where $L\subset H_x$,
  and $K\subset G_x$ denote the respective kernels). 

  We shall now construct an immersion $\widetilde{\iota}\co T\to S$ lifting $\iota\co HT/H\to GS/G$ as in Definition \ref{def:smooth} (\ref{def:smooth:i}) (with the corresponding homomorphism given by the inclusion $H_x\hookrightarrow G_x$). Define
  \[\widetilde{\iota}\co T\hookrightarrow CS\stackrel{\mu^{-1}}{\longrightarrow} C\times S\stackrel{\text{pr}_S}{\longrightarrow}S\]
  and observe that $p_G\circ\widetilde{\iota}=\iota\circ p_H$ on $T$: if $g\in C$, $y\in S$ such that $gy\in T$, then
  \[p_G\circ\widetilde{\iota}(gy)=p_G\circ \text{pr}_S\circ\mu^{-1}(gy)=p_G(y)=p_G(gy)=\iota\circ p_H(gy).\]
  To see that $\widetilde{\iota}$ is an immersion, fix $z\in T$ and let $X\in T_zT$ such that $d\widetilde{\iota}_zX=0$. The latter condition implies $X\in T_z(Gz)$. Since
  \[T_z(Gz)\cap T_zT=T_zN\cap T_z(Gz)\cap T_zT =T_z(Hz)\cap T_zT=\{0\},\]
  we conclude that $X=0$ and hence $\eta$ is an immersion.

  To see that $\eta$ is equivariant with respect to $H_x$, let $h\in H_x$ and $z\in T$. If $g_1,g_2\in C$ and $y_1,y_2\in S$ are such that $z=g_1y_1$ and $hz=g_2y_2$, then $g_2^{-1}hg_1y_1=y_2$ and hence $g_2^{-1}hg_1\in G_x$. By the choice of $C$, we obtain $g_2^{-1}hg_1=h$ and hence
  \[\eta(hz)=y_2=g_2^{-1}hz=h g_1^{-1}z=hy_1=h\eta(z).\]

\end{proof}

\begin{remark}
  Note that we can omit the condition $T_xN\cap T_x(Gx)=T_x(Hx)~\forall x\in N$ in the theorem above if $G$ is discrete or if $N$ is transverse to all $G$-orbits. In either case the condition will be satisfied automatically. With respect to the general case, we admit that we are not aware of any example which could show the necessity of that condition above.
\end{remark}

\begin{corollary}
Let $G$ be a Lie group acting smoothly, almost freely and properly on a manifold $M$. If $H$ is a closed subgroup of $G$, $N$ is an $H$-submanifold of $M$ and $T_xN\cap T_x(Gx)=T_x(Hx)$ for every $x\in N$ and $\iota\co N/H\ni Hx\mapsto Gx\in M/G$ is a topological embedding, then $\iota$ is an orbifold embedding.
\end{corollary}

Note that the corollary above applies in the case that $G$ is finite and $N$ is closed (by Lemma \ref{lemma:embedding}) and also to the case that $N/H$ is compact.

\section{Alternative characterizations of full and embedded suborbifolds}
\label{sec:alt}
In this section we give alternative characterizations of full and embedded suborbifolds to illustrate in which sense our definition of a full suborbifold corresponds to the homonymous notion in \cite{zbMATH06513434} and that the idea of a ``split'' suborbifold in \cite{zbMATH06513434} actually gives just another characterization of an embedded suborbifold. We also explain how to verify two examples of suborbifolds from \cite{zbMATH06513434} which are not full or not embedded, respectively.

The following characterization of full suborbifolds shows that our definition of this term corresponds to the idea of a full suborbifold in \cite{zbMATH06513434}. When denoting isotropy groups using our $\text{Iso}(p)$-notation, we will add the orbifold name as a subscript to avoid ambiguities.

\begin{proposition}
  \label{proposition:full}
  Let $\mathcal{O}$ be an orbifold, $0\le k\le\dim\mathcal{O}$ and let $\mathcal{P}\subset\mathcal{O}$ be a subset. The following are equivalent.
  \begin{enumiii}
  \item \label{proposition:full:1} $\mathcal{P}$ is a $k$-dimensional full suborbifold of $\mathcal{O}$ (in the sense of Definition \ref{def:suborbi} (\ref{def:suborbi:full})).
  \item \label{proposition:full:2} For every $p\in\mathcal{P}$ there is a chart $(U,\widetilde{U}/\Gamma,\pi)$ of $\mathcal{O}$ with the properties that $\Gamma\simeq\Iso_{\mathcal{O}}(p)$ and that there is a $k$-dimensional $\Gamma$-invariant submanifold $\widetilde{V}\subset\widetilde{U}$ such that $\pi(\widetilde{V})$ is an open neighborhood of $p$ in $\mathcal{P}$.
  \end{enumiii}
\end{proposition}

\begin{proof}
  First note that in the local situation of (\ref{proposition:full:2}) $\widetilde{V}$ is a full $\Gamma$-submanifold of $\widetilde{U}$ and hence (\ref{proposition:full:2}) implies (\ref{proposition:full:1}).

  Now assume that $\mathcal{P}$ is a $k$-dimensional full suborbifold. To verify (\ref{proposition:full:2}) let $p\in\mathcal{P}$, let $(U,\widetilde{U}/\Gamma,\pi)$ be a chart of $\mathcal{O}$ around $p$, let $\Delta\subset\Gamma$ be a subgroup and let $\widetilde{V}$ be a $k$-dimensional full $\Delta$-submanifold of $\widetilde{U}$ such that $\pi(\widetilde{V})$ is an open neighborhood of $p$ in $\mathcal{P}$. Let $\widetilde{p}\in\pi^{-1}(p)\cap\widetilde{V}$ and let $\widetilde{U}^\prime$ be a $\Gamma_{\widetilde{p}}$-invariant open neighborhood of $\widetilde{p}$ in $\widetilde{U}$ such that $\widetilde{U}^\prime\cap\gamma\widetilde{U}^\prime=\emptyset$ for every $\gamma\in\Gamma\setminus\Gamma_{\widetilde{p}}$. Setting $\pi^\prime:=\pi_{|\widetilde{U}^\prime}\co\widetilde{U}^\prime\to\pi(\widetilde{U}^\prime)=:U^\prime$ and $\Gamma^\prime:=\Gamma_{\widetilde{p}}$, we obtain a chart $(U^\prime,\widetilde{U}^\prime/\Gamma^\prime,\pi^\prime)$ of $\mathcal{O}$. Then $\widetilde{V}^\prime:=\widetilde{U}^\prime\cap\widetilde{V}$ is a $k$-dimensional submanifold of $\widetilde{U}^\prime$. Since $\Gamma^\prime=\Delta_{\widetilde{p}}$ by Lemma \ref{lemma:stab}, $\widetilde{V}^\prime$ is $\Gamma^\prime$-invariant. Finally note that $\pi^\prime(\widetilde{V}^\prime)=\pi^\prime(\widetilde{U}^\prime\cap\widetilde{V})=U^\prime\cap\pi(\widetilde{V})$ is an open neighborhood of $p$ in $\mathcal{P}$.
\end{proof}

Note that for verifying that a suborbifold is not full it is not sufficient to provide a chart as in Definition \ref{def:suborbi} (\ref{def:suborbi:gen}) with $\widetilde{V}$ a non-full $\Delta$-suborbifold. For instance, the singular point $[0]$ in $\R/\{\pm 1\}$ is a full suborbifold, since $\{0\}$ is a full $\{\pm 1\}$-submanifold of the $\{\pm 1\}$-submanifold $\R$, but $\{0\}$ is not full as a $\{1\}$-submanifold of $\R$. (This example also illustrates that the definition of ``full'' in \cite{zbMATH06513434} depends on the concrete suborbifold structure and not just the subset. It seems that the notion of ``canonical structure'' defined in \cite{2016arXiv160605902B} may help overcome that ambiguity.)

The following proposition gives a useful criterion based on ideas already used in \cite[Example 10]{zbMATH06513434} (see the example below).

\begin{proposition}
  \label{prop:notfull}
  Let $\mathcal{O}$ be an orbifold and let $\mathcal{P}$ be a full suborbifold of $\mathcal{O}$. Moreover, let $(U,\widetilde{U}/\Gamma,\pi)$ be a chart on $\mathcal{O}$ such that $\Gamma$ is abelian, let $\Delta$ be a subgroup of $\Gamma$ and $\widetilde{V}$ a (not necessarily full) connected $\Delta$-submanifold of $\widetilde{U}$ such that $\pi(\widetilde{V})$ is an open subset of $\mathcal{P}$. Then for $\Omega:=\{\gamma\in\Gamma;~\gamma x=x~\forall x\in\widetilde{V}\}$, $p\in V$ and $\widetilde{p}\in\pi^{-1}(p)\cap \widetilde{V}$,
  there is an isomorphism
  \[\Iso_{\mathcal{P}}(p)\simeq\Gamma_{\widetilde{p}}/\Omega.\]
\end{proposition}
\begin{proof}
  By Proposition \ref{proposition:full}, there is a chart $(U^\prime,\widetilde{U}^\prime/\Gamma^\prime,\pi^\prime)$ of $\mathcal{O}$ such that $\Gamma^\prime\simeq\Iso_{\mathcal{O}}(p)$ and there is a $k$-dimensional $\Gamma^\prime$-invariant submanifold $\widetilde{V}^\prime\subset\widetilde{U}^\prime$ such that $\pi(\widetilde{V}^\prime)$ is an open neighborhood of $p$ in $\mathcal{P}$. Choosing $U^\prime$ sufficiently small, we can assume that $U^\prime\subset U$ and $V^\prime\subset V$ and that there is an injection $\mu\co (U^\prime,\widetilde{U}^\prime/\Gamma^\prime,\pi^\prime)\to (U,\widetilde{U}/\Gamma,\pi)$. Substituting $\widetilde{V}^\prime$ by the connected component of the unique $\widetilde{p}^\prime\in\pi^{-1}(p)$ in $\widetilde{V}^\prime$, we can assume that $\widetilde{V}^\prime$ is connected. Then $\Iso_{\mathcal{P}}(p)\simeq \Gamma^\prime/K^\prime$ (with $K^\prime\subset\Gamma^\prime$ denoting the kernel of the $\Gamma^\prime$-action on $\widetilde{V}^\prime$). Composing $\mu$ by an appropriate element of $\Gamma$ if necessary, we can guarantee that $\mu(\widetilde{V}^\prime)$ contains $\widetilde{p}$.

Since $\mu(\widetilde{U}^\prime)$ 
contains $\widetilde{p}$, we easily verify that $\overline{\mu}(\Gamma^\prime)=\Gamma_{\widetilde{p}}$. To see that $\overline{\mu}(K^\prime)=\Omega$, first let $\gamma\in\overline{\mu}(K^\prime)$ and equip $\widetilde{U}$ with an $\Gamma$-invariant Riemannian metric. Then $\gamma$ fixes $\mu(\widetilde{V}^\prime)$ and, since $\Gamma$ is abelian, it also fixes the open subset $\pi^{-1}(V^\prime)\cap\widetilde{V}$ of $\widetilde{V}$. Since $\widetilde{V}$ is connected and $\gamma$ is an isometry, we obtain $\gamma\in\Omega$. On the other hand, let $\gamma\in\Omega\subset\Gamma_{\widetilde{p}}=\overline{\mu}(\Gamma^\prime)$ and $\gamma^\prime\in\Gamma^\prime$ such that $\gamma=\overline{\mu}(\gamma^\prime)$. Since $\Gamma$ is abelian and $V^\prime\subset V$, the transformation $\gamma$ fixes $\mu(\widetilde{V}^\prime)$ and hence $\gamma^\prime\in K^\prime$. Finally, we conclude that $\Iso_{\mathcal{P}}(p)\simeq\Gamma_{\widetilde{p}}/\Omega$. 
\end{proof}

\begin{example}[\cite{zbMATH06513434} Example 10]
  Let $M=\R\times\R$, $G=\{\pm 1\}\times\{\pm 1\}$, $N=\{(x,x);~x\in\R\}$ and $H=\{(1,1),(-1,-1)\}$. Then $\mathcal{P}=N/H$ is a suborbifold of $\mathcal{O}=M/G$, $G_{(0,0)}=G\simeq\Z_2\times\Z_2$ and $\Omega=\{g\in G;~g(x,x)=(x,x)~\forall (x,x)\in N\}=\{(1,1)\}$, but $\Iso_{\mathcal{P}}([(0,0)])\simeq H\simeq\Z_2$ is not isomorphic to $G_{(0,0)}/\Omega$. By Proposition \ref{prop:notfull}, $\mathcal{P}$ is not full in $\mathcal{O}$.
\end{example}

We shall now verify an example of a full suborbifold which is not embedded. Of course, to see that a suborbifold is not embedded, it is not sufficient to find a chart as in Definition \ref{def:suborbi} (\ref{def:suborbi:gen}) with $\Delta$ acting non-effectively (as can be illustrated again by the singular point $[0]$ in $\R/\{\pm 1\}$). The example below is taken from \cite[Example 12]{zbMATH06513434}.

\begin{example}
  Let $M:=\C^2$ equipped with the effective action by the group $G\simeq\Z_4$ generated by
  $\begin{pmatrix}
    i&0\\
    0&-1
  \end{pmatrix}$. Then $N:=\{0\}\times\C$ is a full $G$-submanifold of $M$ and hence $\mathcal{P}:=N/G$ a full suborbifold of $\mathcal{O}:=M/G$.

  Suppose that $\mathcal{P}$ is embedded. Then there is a chart $(U_1,\widetilde{U}_1/\Gamma_1,\pi_1)$ around $[(0,0)]$ on $\mathcal{P}$ and a chart $(U_2,\widetilde{U}_2/\Gamma_2,\pi_2)$ around $[(0,0)]$ on $\mathcal{O}$ together with a smooth embedding $\widetilde{f}\co\widetilde{U}_1\to\widetilde{U}_2$ and a homomorphism $\Theta\co\Gamma_1\to\Gamma_2$ such that $\pi\circ\widetilde{f}=\pi^\prime$ and $\widetilde{f}(\gamma x)=\Theta(\gamma)\widetilde{f}(x)$ for every $\gamma\in\Gamma_1$, $x\in\widetilde{U}_1$. Note that the latter condition implies that $\Theta$ is injective. Diminishing $U_1$ and $U_2$ if necessary, we can assume that there is an injection $\lambda$ from the local chart $(U_2,\widetilde{U}_2/\Gamma_2,\pi_2)$ into the global chart given by the quotient map $M\to M/G$ and 
that $\Gamma_1\simeq\Iso_{\mathcal{P}}([(0,0)])\simeq \Z_2$. 
Then $\overline{\lambda}(\Theta(\Gamma_1))\simeq\Z_2$ is the subgroup of $G\simeq\Z_4$ generated by $\begin{pmatrix}
    -1&0\\
    0&1
  \end{pmatrix}$. But this contradicts the fact that $\overline{\lambda}(\Theta(\Gamma_1))$ acts effectively on the open neighborhood $\lambda(\widetilde{f}(\widetilde{U}_1))$ of $(0,0)$ in $N$.
\end{example}

We should note that  \cite[Section 5]{zbMATH06513434} already contains the claim that the suborbifold above is not an image of a ``complete orbifold embedding'', apparently based on \cite[Theorem 1 (2)]{zbMATH06513434}. However, the property of being ``split'' in that theorem is only obvious for the suborbifold chart around each point in the image which appears in the definition of the embedding. Thus it is not clear if providing one (global) suborbifold chart which is not ``split'', as has been done in \cite[Example 12]{zbMATH06513434}, is sufficient for concluding that a suborbifold is not the image of an embedding.

We now give an alternative characterization of embedded suborbifolds, which shows that the idea of a ``split'' (and ``saturated'') suborbifold in \cite[Definition 6]{zbMATH06513434} corresponds to our notion of an embedded orbifold. The result resembles \cite[Theorem 1 (2)]{zbMATH06513434}.


\begin{proposition}
  \label{proposition:split}
  Let $\mathcal{O}$ be an orbifold, $0\le k\le\dim\mathcal{O}$ and $\mathcal{P}\subset\mathcal{O}$ a subset. The following are equivalent.
  \begin{enumiii}
  \item \label{proposition:embedded:1} $\mathcal{P}$ is a $k$-dimensional embedded suborbifold of $\mathcal{O}$ (in the sense of Definition \ref{def:suborbi} (\ref{def:suborbi:embedded})).
  \item \label{proposition:embedded:2} For every $p\in\mathcal{P}$ there is a chart $(U,\widetilde{U}/\Gamma,\pi)$ of $\mathcal{O}$ with the following property: there is a subgroup $\Delta\subset\Gamma$ and a connected $k$-dimensional $\Delta$-submanifold $\widetilde{V}\subset\widetilde{U}$ such that $\pi(\widetilde{V})$ is an open neighborhood of $p$ in $\mathcal{P}$ and, with $K$ denoting the kernel of the $\Delta$-action on $\widetilde{V}$, the canonical projection $\Delta\to\Delta/K$ has a homomorphic right inverse.
  \end{enumiii}
\end{proposition}

\begin{proof}
  Given an embedded suborbifold $\mathcal{P}$ and $p\in\mathcal{P}$, let $(U,\widetilde{U}/\Gamma,\pi)$, $\Delta$, $\widetilde{V}$ be as in Definition \ref{def:suborbi} (\ref{def:suborbi:embedded}). Since the kernel $K$ of the $\Delta$-action on $\widetilde{V}$ is trivial, $(U,\widetilde{U}/\Gamma,\pi)$ is a chart as in (\ref{proposition:embedded:2}) above. Hence (\ref{proposition:embedded:1}) implies (\ref{proposition:embedded:2}).

  Now let $\mathcal{P}$ satisfy (\ref{proposition:embedded:2}) and let $p\in\mathcal{P}$. Let $(U,\widetilde{U}/\Gamma,\pi)$, $\Delta$, $K$ and $\widetilde{V}$ be as (\ref{proposition:embedded:2}) and let $\sigma\co\Delta/K\to\Delta$ be a homomorphic right inverse of the projection $q=[\cdot]\co\Delta\to\Delta/K$. Let $\Delta^\prime$ denote the image of $\sigma$. To see that $\widetilde{V}$ is a $\Delta^\prime$-submanifold of $\widetilde{U}$, first note that $\widetilde{V}$ is obviously $\Delta^\prime$-invariant. If $\gamma\in\Gamma$, $x\in\widetilde{V}$ such that $\gamma x\in\widetilde{V}$, then there is $\delta\in\Delta$ such that $\gamma x=\delta x$ and hence $\sigma([\delta])x=q(\sigma([\delta]))x=[\delta]x=\delta x=\gamma x$.
  To see that $\Delta^\prime$ acts effectively on $\widetilde{V}$, let $\delta\in\Delta$ such that $\sigma([\delta])x=x$ for every $x\in\widetilde{V}$. Then $[\delta]x=q(\sigma([\delta]))x=\sigma([\delta])x=x$ for every $x\in\widetilde{V}$. Since $\Delta/K$ acts effectively on $\widetilde{V}$, we obtain $\sigma([\delta])=\sigma([e])=e$. Hence $\Delta^\prime$ acts effectively on $\widetilde{V}$ and we conclude that $\mathcal{P}$ is an embedded suborbifold.
\end{proof}

We should mention that the table in Section 5 of \cite{zbMATH06513434} suggests that Example 13 in that article provides suborbifolds in some broader sense to which our proposition above cannot be generalized, referred to as non-''saturated'' ``suborbifolds'' in \cite{zbMATH06513434}. However, that example does not specify how the (underlying spaces of the) alleged suborbifolds should be seen as subsets of each other and hence it is not clear if part (1) of the suborbifold definition (\cite[Definition 4]{zbMATH06513434}) is satisfied. 


\section{Transversality}
\label{sec:transversality}

We will now generalize classical notions of transversality to orbifolds. For basic results on transversality in the context of manifolds we refer the reader to \cite{zbMATH03562121}.

\begin{definition}
  \label{def:trans}
  Let $\mathcal{O}$ be an orbifold. Two full suborbifolds $\mathcal{P}_1,\mathcal{P}_2\subset\mathcal{O}$ are called \emph{transverse} if for every $p\in \mathcal{P}_1\cap\mathcal{P}_2$ there is a chart $(U,\widetilde{U}/\Gamma,\pi)$ of $\mathcal{O}$,  and for every $i=1,2$ there is a subgroup $\Delta_i\subset\Gamma$ and a full $\Delta_i$-submanifold $\widetilde{V}_i$ such that $\pi(\widetilde{V}_i)$ is an open subset of $\mathcal{P}_i$ and such that $\widetilde{V}_1$, $\widetilde{V}_2$ are transverse submanifolds intersecting in some point of $\pi^{-1}(p)$.
\end{definition}

\begin{example}
  \begin{enumiii}
  \item If $\mathcal{O}_1, \mathcal{O}_2$ are orbifolds and $\mathcal{P}_1\subset\mathcal{O}_1,\mathcal{P}_2\subset\mathcal{O}_2$ are full suborbifolds, then $\mathcal{P}_1\times\mathcal{O}_2$ and $\mathcal{O}_1\times\mathcal{P}_2$ are easily seen to be transverse suborbifolds of $\mathcal{O}_1\times \mathcal{O}_2$. 
  \item If $G$ is a discrete group acting smoothly and properly on a manifold $M$, $\Gamma$ is a subgroup of $G$ and $N_1$, $N_2$ are transverse full closed $\Gamma$-submanifolds of $M$, then $N_1/\Gamma$ and $N_2/\Gamma$ are transverse full suborbifolds of $M/G$: If $p\in (N_1/\Gamma) \cap (N_2/\Gamma)\subset M/G$, let $\widetilde{p}\in N_1\cap N_2$ such that $G\widetilde{p}=p$. Let $\widetilde{U}$ be a slice at (i.e., an open full $G_{\widetilde{p}}$-submanifold of $M$ containing) $\widetilde{p}$. Then $\pi\co \widetilde{U}\to\widetilde{U}/G_{\widetilde{p}}=:U\subset M/G$ defines a chart on $M/G$ and $\Delta_i:=\Gamma_{\widetilde{p}}=G_{\widetilde{p}}\cap\Gamma$, $\widetilde{V}_i:=\widetilde{U}\cap N_i$, $i=1,2$, satisfy the conditions from \ref{def:trans}.
  \end{enumiii}
\end{example}

\begin{theorem}
  \label{theorem:inter}
  If $\mathcal{O}$ is an $n$-dimensional orbifold and $\mathcal{P}_1,\mathcal{P}_2\subset\mathcal{O}$ are transverse full suborbifolds of dimension $k_1$, $k_2$, respectively, then $\mathcal{P}_1\cap\mathcal{P}_2$ is a full suborbifold of $\mathcal{O}$ of dimension $k_1+k_2-n$. 
\end{theorem}
\begin{proof}
  Let $p\in\mathcal{P}_1\cap\mathcal{P}_2$ and let $(U,\widetilde{U}/\Gamma,\pi)$, $\Delta_i$, $\widetilde{V}_i$, $i=1,2$, be as in Definition \ref{def:trans}. Since $\widetilde{V}_1$, $\widetilde{V}_2$ are transverse submanifolds of dimension $k_1$, $k_2$, respectively, the non-empty intersection $\widetilde{V}:=\widetilde{V}_1\cap \widetilde{V}_2$ is a submanifold of dimension $k_1+k_2-n$ (see \cite{zbMATH03562121}).

Note that $\widetilde{V}$ is invariant under $\Delta:=\Delta_1\cap\Delta_2\subset\Gamma$. To conclude that $\widetilde{V}\subset\widetilde{U}$ is a full $\Delta$-submanifold, let $\gamma\in\Gamma$ and $x\in\widetilde{V}$ such that $\gamma x\in\widetilde{V}$. Since each $\widetilde{V}_i$ is a full $\Delta_i$-submanifold, we obtain $\gamma\in\Delta$.
\end{proof}

\begin{definition}
  \label{def:transmap}
  Let $f\co \mathcal{O}_1\to\mathcal{O}_2$ be a smooth map between orbifolds and let $\mathcal{Q}\subset\mathcal{O}_2$ be a $k$-dimensional full suborbifold. We say that $f$ is \emph{transverse} to $\mathcal{Q}$ (and write $f\pitchfork \mathcal{Q}$) if for every $p\in f^{-1}(\mathcal{Q})$ there are:
  \begin{itemize}
  \item a chart $(U_1,\widetilde{U}_1/\Gamma_1,\pi_1)$ on $\mathcal{O}_1$ around $p$,
  \item a chart $(U_2,\widetilde{U}_2/\Gamma_2,\pi_2)$ on $\mathcal{O}_2$ around $f(p)$ such that $\Gamma_2\simeq\Iso(f(p))$ and $f(U_1)\subset U_2$,
  \item a $k$-dimensional $\Gamma_2$-invariant submanifold $\widetilde{V}$ of $\widetilde{U}_2$ such that $\pi_2(\widetilde{V})$ is an open neighborhood of $f(p)$ in $\mathcal{Q}$,
  \item a smooth map $\widetilde{f}\co \widetilde{U}_1\to\widetilde{U}_2$ transverse to $\widetilde{V}$ such that $\pi_2\circ\widetilde{f}=f\circ\pi_1$,
  \item a homomorphism $\Theta\co \Gamma_1\to\Gamma_2$ such that $\widetilde{f}(\gamma x)=\Theta(\gamma)\widetilde{f}(x)$ for every $\gamma\in\Gamma_1$, $x\in\widetilde{U}_1$.
  \end{itemize}
\end{definition}

\begin{theorem}
  \label{theorem:preimage}
  Let $f\co \mathcal{O}_1\to\mathcal{O}_2$ be a smooth map between orbifolds and let $\mathcal{Q}\subset\mathcal{O}_2$ be a full suborbifold such that $f\pitchfork \mathcal{Q}$. Then $\mathcal{P}:=f^{-1}(\mathcal{Q})$ is empty or a full suborbifold of $\mathcal{O}_1$ of dimension $\dim\mathcal{O}_1-(\dim\mathcal{O}_2-\dim\mathcal{Q})$.
\end{theorem}
\begin{proof}
  Given $p\in \mathcal{P}$, let $(U_i,\widetilde{U}_i/\Gamma_i,\pi_i)$, $i=1,2$, $\widetilde{V}\subset\widetilde{U}_2$, $\widetilde{f}\co\widetilde{U}_1\to\widetilde{U}_2$ and $\Theta\co\Gamma_1\to\Gamma_2$ be as in Definition \ref{def:transmap}. Since $\widetilde{f}$ is transverse to $\widetilde{V}$ and $\widetilde{V}^\prime:=\widetilde{f}^{-1}(\widetilde{V})$ contains $\pi_1^{-1}(p)$, $\widetilde{V}^\prime$ is a submanifold of $\widetilde{U}_1$ of dimension $\dim\mathcal{O}_1-(\dim\mathcal{O}_2-\dim\mathcal{Q})$ (see \cite{zbMATH03562121}).
  Since $\pi_2(\widetilde{V})$ is an open neighborhood of $f(p)$ in $\mathcal{Q}$, the intersection $f^{-1}(\pi_2(\widetilde{V}))\cap U_1$ is an open neighborhood of $p$ in $\mathcal{P}$. Since $\widetilde{V}$ is $\Gamma_2$-invariant, we easily verify that $\pi_1(\widetilde{V}^\prime)=f^{-1}(\pi_2(\widetilde{V}))\cap U_1$ and that $\widetilde{V}^\prime$ is $\Gamma_1$-invariant. Since $p\in \mathcal{P}$ was arbitrary, we conclude that $\mathcal{P}$ is a full suborbifold of $\mathcal{O}_1$ of dimension $\dim\mathcal{O}_1-(\dim\mathcal{O}_2-\dim\mathcal{Q})$.
\end{proof}

Proposition \ref{proposition:full} implies that an orbifold submersion is transverse to every full suborbifold of its codomain and hence we obtain the following corollary.

\begin{corollary}
  \label{cor:subm}
    Let $f\co \mathcal{O}_1\to\mathcal{O}_2$ be a submersion between orbifolds and let $\mathcal{Q}\subset\mathcal{O}_2$ be a full suborbifold. Then $f^{-1}(\mathcal{Q})$ is either empty or a full suborbifold of $\mathcal{O}_1$ of dimension $\dim\mathcal{O}_1-(\dim\mathcal{O}_2-\dim\mathcal{Q})$.
\end{corollary}

\begin{example}
  Given orbifolds $\mathcal{O}_1$, $\mathcal{O}_2$ and submersions $f_1\co\mathcal{O}_1\to M$, $f_2\co\mathcal{O}_2\to M$ into a manifold $M$, applying Corollary \ref{cor:subm} to the submersion $f:=f_1\times f_2$ and the diagonal $\mathcal{Q}$ in $M\times M$, we conclude that the fibered product $\mathcal{O}_1\tensor[_{f_1}]{\times}{_{f_2}} \mathcal{O}_2=\{(p_1,p_2)\in\mathcal{O}_1\times\mathcal{O}_2;~f_1(p_1)=f_2(p_2)\}$ is either empty or a full suborbifold of $\mathcal{O}_1\times\mathcal{O}_2$ of dimension $\dim \mathcal{O}_1+\dim\mathcal{O}_2-\dim M$.
\end{example}

Applying Theorem \ref{theorem:preimage} to a full suborbifold given by just a point (see Example \ref{example:suborbis} (\ref{example:suborbis:pt})), we obtain the regular value theorem below. Note that \cite[Theorem 4.2]{zbMATH06100632} looks almost identical but that paper does not impose any compatibility conditions on the charts on a ``suborbifold''. Without a result like our Proposition \ref{proposition:suborbis} it is not clear if the preimage in \cite[Theorem 4.2]{zbMATH06100632} carries an orbifold structure.


\begin{corollary}[Regular value theorem]
  Let $f\co\mathcal{O}_1\to\mathcal{O}_2$ be a smooth map between orbifolds and let $q\in f(\mathcal{O}_1)$ such that $\rk f(p)=\dim\mathcal{O}_2$ for every $p\in f^{-1}(q)$. Then $f^{-1}(q)$ is a full suborbifold of $\mathcal{O}_1$ of dimension $\dim\mathcal{O}_1-\dim\mathcal{O}_2$.
\end{corollary}

\section*{Acknowledgements}
We would like to thank Francisco J. Gozzi for various helpful discussions on Section 2 and for his idea to use reference \cite{claudio} in the proof of Proposition \ref{proposition:suborbis}. We also thank Dorothee Sch\"uth, Marcos Alexandrino and Matias del Hoyo for suggesting various improvements to earlier drafts of this article. We would also like to thank the referee for thorough proofreading and various helpful suggestions.

\bibliographystyle{elsarticle-num}


\end{document}